\documentclass[12pt,letterpaper]{amsart}
\usepackage{amsmath,amsfonts,amssymb}

\usepackage{longtable}
\usepackage[mathscr]{eucal}
\usepackage{amsmath, amsthm}
\usepackage{mathrsfs}
\usepackage{amsbsy}
\usepackage{dsfont}
\usepackage{bbm}
\usepackage{wasysym}
\usepackage{stmaryrd}
\usepackage{url}

\theoremstyle{plain}
\usepackage{color}
\usepackage[breaklinks]{hyperref}

\input xypic
\xyoption{all}

\makeindex
\makeglossary

\begin{document}
\baselineskip = 16pt

\newcommand \ZZ {{\mathbb Z}}
\newcommand \NN {{\mathbb N}}
\newcommand \RR {{\mathbb R}}
\newcommand \PR {{\mathbb P}}
\newcommand \AF {{\mathbb A}}
\newcommand \GG {{\mathbb G}}
\newcommand \QQ {{\mathbb Q}}
\newcommand \CC {{\mathbb C}}
\newcommand \bcA {{\mathscr A}}
\newcommand \bcC {{\mathscr C}}
\newcommand \bcD {{\mathscr D}}
\newcommand \bcF {{\mathscr F}}
\newcommand \bcG {{\mathscr G}}
\newcommand \bcH {{\mathscr H}}
\newcommand \bcM {{\mathscr M}}
\newcommand \bcI {{\mathscr I}}
\newcommand \bcJ {{\mathscr J}}
\newcommand \bcK {{\mathscr K}}
\newcommand \bcL {{\mathscr L}}
\newcommand \bcO {{\mathscr O}}
\newcommand \bcP {{\mathscr P}}
\newcommand \bcQ {{\mathscr Q}}
\newcommand \bcR {{\mathscr R}}
\newcommand \bcS {{\mathscr S}}
\newcommand \bcV {{\mathscr V}}
\newcommand \bcU {{\mathscr U}}
\newcommand \bcW {{\mathscr W}}
\newcommand \bcX {{\mathscr X}}
\newcommand \bcY {{\mathscr Y}}
\newcommand \bcZ {{\mathscr Z}}
\newcommand \goa {{\mathfrak a}}
\newcommand \gob {{\mathfrak b}}
\newcommand \goc {{\mathfrak c}}
\newcommand \gom {{\mathfrak m}}
\newcommand \gon {{\mathfrak n}}
\newcommand \gop {{\mathfrak p}}
\newcommand \goq {{\mathfrak q}}
\newcommand \goQ {{\mathfrak Q}}
\newcommand \goP {{\mathfrak P}}
\newcommand \goM {{\mathfrak M}}
\newcommand \goN {{\mathfrak N}}
\newcommand \uno {{\mathbbm 1}}
\newcommand \Le {{\mathbbm L}}
\newcommand \Spec {{\rm {Spec}}}
\newcommand \Gr {{\rm {Gr}}}
\newcommand \Pic {{\rm {Pic}}}
\newcommand \Jac {{{J}}}
\newcommand \Alb {{\rm {Alb}}}
\newcommand \Corr {{Corr}}
\newcommand \Chow {{\mathscr C}}
\newcommand \Sym {{\rm {Sym}}}
\newcommand \Prym {{\rm {Prym}}}
\newcommand \cha {{\rm {char}}}
\newcommand \eff {{\rm {eff}}}
\newcommand \tr {{\rm {tr}}}
\newcommand \Tr {{\rm {Tr}}}
\newcommand \pr {{\rm {pr}}}
\newcommand \ev {{\it {ev}}}
\newcommand \cl {{\rm {cl}}}
\newcommand \interior {{\rm {Int}}}
\newcommand \sep {{\rm {sep}}}
\newcommand \td {{\rm {tdeg}}}
\newcommand \alg {{\rm {alg}}}
\newcommand \im {{\rm im}}
\newcommand \gr {{\rm {gr}}}
\newcommand \op {{\rm op}}
\newcommand \Hom {{\rm Hom}}
\newcommand \Hilb {{\rm Hilb}}
\newcommand \Sch {{\mathscr S\! }{\it ch}}
\newcommand \cHilb {{\mathscr H\! }{\it ilb}}
\newcommand \cHom {{\mathscr H\! }{\it om}}
\newcommand \colim {{{\rm colim}\, }} 
\newcommand \End {{\rm {End}}}
\newcommand \coker {{\rm {coker}}}
\newcommand \id {{\rm {id}}}
\newcommand \van {{\rm {van}}}
\newcommand \spc {{\rm {sp}}}
\newcommand \Ob {{\rm Ob}}
\newcommand \Aut {{\rm Aut}}
\newcommand \cor {{\rm {cor}}}
\newcommand \Cor {{\it {Corr}}}
\newcommand \res {{\rm {res}}}
\newcommand \red {{\rm{red}}}
\newcommand \Gal {{\rm {Gal}}}
\newcommand \PGL {{\rm {PGL}}}
\newcommand \Bl {{\rm {Bl}}}
\newcommand \Sing {{\rm {Sing}}}
\newcommand \spn {{\rm {span}}}
\newcommand \Nm {{\rm {Nm}}}
\newcommand \inv {{\rm {inv}}}
\newcommand \codim {{\rm {codim}}}
\newcommand \Div{{\rm{Div}}}
\newcommand \CH{{\rm{CH}}}
\newcommand \sg {{\Sigma }}
\newcommand \DM {{\sf DM}}
\newcommand \Gm {{{\mathbb G}_{\rm m}}}
\newcommand \tame {\rm {tame }}
\newcommand \znak {{\natural }}
\newcommand \lra {\longrightarrow}
\newcommand \hra {\hookrightarrow}
\newcommand \rra {\rightrightarrows}
\newcommand \ord {{\rm {ord}}}
\newcommand \Rat {{\mathscr Rat}}
\newcommand \rd {{\rm {red}}}
\newcommand \bSpec {{\bf {Spec}}}
\newcommand \Proj {{\rm {Proj}}}
\newcommand \pdiv {{\rm {div}}}
\newcommand \wt {\widetilde }
\newcommand \ac {\acute }
\newcommand \ch {\check }
\newcommand \ol {\overline }
\newcommand \Th {\Theta}
\newcommand \cAb {{\mathscr A\! }{\it b}}

\newenvironment{pf}{\par\noindent{\em Proof}.}{\hfill\framebox(6,6)
\par\medskip}

\newtheorem{theorem}[subsection]{Theorem}
\newtheorem{conjecture}[subsection]{Conjecture}
\newtheorem{proposition}[subsection]{Proposition}
\newtheorem{lemma}[subsection]{Lemma}
\newtheorem{remark}[subsection]{Remark}
\newtheorem{remarks}[subsection]{Remarks}
\newtheorem{definition}[subsection]{Definition}
\newtheorem{corollary}[subsection]{Corollary}
\newtheorem{example}[subsection]{Example}
\newtheorem{examples}[subsection]{examples}

\title{Beilinson's conjecture on  K3 surfaces with an involution}
\author{Kalyan Banerjee}

\email{kalyan.ba@srmap.edu.in}

\begin{abstract}
In this note, we prove that the Beilinson conjecture holds for certain examples of K3 surfaces  over $\bar {\mathbb{Q}}$  equipped with an involution, when the quotient of the surface by the involution is the projective plane branched along a sextic.
\end{abstract}
\maketitle

\section{Introduction}

One of the most important problems in algebraic geometry is to compute the Chow groups for higher codimensional cycles on a smooth projective variety. The theorem due to Mumford in \cite{M}, proves that the Chow group of zero cycles on a smooth projective surface over complex numbers is infinite-dimensional, provided that the geometric genus of the surface is greater than zero, in the sense that the natural maps from symmetric powers to the Chow group of zero cycles are never surjective. The converse is a conjecture due to Bloch, that is, if we have a smooth projective complex surface of geometric genus zero, then the Chow group of zero cycles is isomorphic to the group of points on an abelian variety.

This theorem has been proved for surfaces not of general type with geometric genus zero due to \cite{BKL}. The case for surfaces of general type, that is, surfaces of Kodaira dimension 2 with geometric genus zero, is still open. Some examples of such surfaces for which the Bloch's conjecture holds are due to \cite{B}, \cite{IM}, \cite{PW},\cite{V},\cite{VC}.

The situation over $\bar{\QQ}$ is completely different. The conjecture due to Beilinson states that for any smooth projective surface defined over $\bar{\QQ}$, the albanese kernel is trivial, that is, the Chow group of zero cycles of degree zero is isomorphic to the albanese variety.

This paper concerns the proof of Beilinson's conjecture for some class of K3 surfaces equipped with an involution, when the quotient of the surface is the projective plane $\PR^2$ and the branch locus is a sextic curve in $\PR^2$.

The technique due to \cite{Voi}, tells us that if we have  an involution on a complex K3 surface which acts as identity on the two form on the surface,  then the involution acts as identity on the $A_0$ (zero cycles of degree zero modulo rational equivalence) of the K3 surface. This has also been proved for some classes of K3 surfaces due to \cite{GT},\cite{HK}.

We prove that the involution acts as an identity on the group $A_0(S)$, for the K3 surfaces  $S$ over $\bar \QQ$ such that the quotient is $\PR^2$ branched along a sextic and $S$ is equipped with an involution. Since the surface of the quotient is $\PR^2$ which has  trivial $A_0$,  the involution acts as $-1$ in $A_0(S)$. Together, these two information tell us that the group $A_0(S)$ is trivial by Roitman's torsion theorem \cite{R2}.

The main difference of our approach from that of Voisin present in \cite{Voi}, is that we need to perform the techniques of Voisin over a countable algebraically closed field and develop the necessary details for it.

The main result of this paper is the following.

\begin{theorem}
Let $S$ be a  K3 surface defined over $\bar \QQ$ such that there is an involution $i$ on $S$ and $S/i$ is isomorphic to the projective plane $\PR^2$. In addition, the map $S\to S/i$ is branched along a sextic curve in $\PR^2$. Suppose that $S$ contains infinitely many rational curves.  Then the Beilinson's conjecture holds on $S$.
\end{theorem}

{\small \textbf{Acknowledgements:} The author thanks SRM AP for hosting this project. }

\section{Finite dimensionality in the sense of Roitman and $\CH_0$}

First, we recall the notion of finite dimensionality in the sense of Roitman \cite{R1}.

\begin{definition}
\label{defn1}
Let $X$ be a smooth projective variety on $\bar{\mathbb{Q}}$ and let $P$ be a subgroup of $\CH_0(X)$, then we will say that $P$ is finite-dimensional in the Roitman sense, if there exist a smooth projective variety $W$ and a correspondence $\Gamma$ on $W\times X$, such that $P$ is contained in the set $\{\Gamma_*(w)|w\in W\}$.
\end{definition}

The following proposition is due to Roitman \cite{R1} and also to Voisin \cite{Voi} over $\CC$. We give a proof of it on $\bar{\QQ}$. This is crucial as we are working on countable algebraically closed fields.

\begin{theorem}
\label{theorem 1}
Let $S$ be a smooth projective surface over $\bar{\QQ}$ of zero irregularity consisting of infinitely many rational curves. Let $Z$ be a correspondence on $S\times S$. Assume that the image of $Z_*$ from $\CH_0(S)_{hom}$ to $\CH_0(S)$ is finite-dimensional in the Roitman sense. Then the map $Z_*$ from ${\CH_0(S)}_{hom}$ to $\CH_0(S)$ is the zero map.
\end{theorem}

\begin{proof}
By the definition of finite dimensionality in the sense of Roitman, there exist a smooth projective variety $W$ and a correspondence $\Gamma$ on $W\times S$, such that the image of $Z_*$ is contained in the set $\Gamma_*(W)$. Let $C\subset S$, be a general smooth projective curve obtained by taking a smooth hyperplane section of $S$.  Now we have the following lemma about $J(C)$.

\begin{lemma}
The abelian variety $J(C)$ is simple for a general hyperplane section.
\end{lemma}
\begin{proof}
Suppose that $J(C)$ has a non-zero proper abelian subvariety $A$ for a general $C$. Then, extending the scalars, we have the geometric generic Jacobian $J(C_{\bar \eta})$ ($\bar\eta$ is the geometric generic points of the open subset $U$ of $\PR^n$ that parametrize the smooth hyperplane sections of $S$. We fix an embedding $S\to \PR^n$). Also we have 
$$A_{\bar\eta}=A\times_{\bar \QQ} \bar\eta$$
is a sub-abelian variety of $J(C_{\bar\eta})$. 

Spreading out both abelian varieties, we have two abelian fibrations $\bcA\to U, \bcJ\to U$, such that the geometric generic fibers are $A_{\bar\eta}, J(C_{\bar\eta})$. Note that the family $\bcA\to U$ is iso-trivial. These two fibrations define two local systems whose stems are $H^1(A_{\bar\eta},\QQ_l)$, $H^1(J(C_{\bar\eta}),\QQ_l)$, and the above cohomologies admit a representation of $\pi_1(U,\bar\eta)$ (the 'etale fundamental group of $U$ at $\bar\eta$).

We know by Picard-Lefshcetz formula that the $\QQ_l$ vector space $H^1(J(C_{\bar\eta}),\QQ_l)$ admits an irreducible representation of the 'etale fundamental group and $H^1(A_{\bar\eta},\QQ_l)$ is a submodule invariant under the action of the fundamental group. So it must be $\{0\}$ or all of $H^1(J(C_{\bar\eta}),\QQ_l)$. Then the Tate module of $A_{\bar\eta}$ is either zero or all of the Tate module of $J(C_{\bar\eta})$. By the Proof of Tate conjecture (due to Faltings and Zarhin), it follows that $A_{\bar\eta}=0$ or $J(C_{\bar\eta})$. Consequently, we have $A=0$ or $J(C)$. In the latter case, if $A$ is not equal to $J(C)$, then it's dimension is strictly less than that of $J(C)$. But 
$$\dim(A)=\dim(A_{\bar\eta})=\dim(J(C_{\bar\eta}))=\dim(J(C))\;.$$
Therefore, $A=J(C)$ for a general $C$ in the linear system.
\end{proof}

We now have such a $C$ inside $S$. Let $j:C\to S$ denote the closed embedding of $C$ in $S$. By the above lemma $J(C)$ is simple. So we have a homomorphism $j_*$ from $J(C)$ to $\CH_0(S)$. Note that for $H_i$ sufficiently ample we can make the dimension of $J(C)$ arbitrarily large so that the dimension of $J(C)$ is greater than that of $W$.

Let $R$ be a set inside $J(C)\times W$, given by

$$\{(k,w)|Z_*j_*(k)=\Gamma_*(w)\}$$

Due to Roitman, we have:

\begin{lemma}
$R$ is a countable union of Zariski-closed subsets in the product.
\end{lemma}

\begin{lemma}
The homomorphism $Z_*j_*$ is zero on $J(C)$.
\end{lemma}

\begin{proof}
Consider the kernel of $j_*$. By Roitman's theorem on torsion points, it contains all torsions in the Jacobian $J(C)$. The torsion points are Zariski dense in the Jacobian, therefore $ker(j_*)$ is Zariski dense as well. Now, by the following argument, the kernel contains infinitely many points of infinite order.

Let $R_i$ be a countable collection of rational curves on $S$ (possibly reducible), then we can take $C$ to be such that $C.R_i\geq 2$, because $C$ is smooth and ample. 

Therefore, $C\times C$ contains infinitely many pairs of points such that the difference of points $P-Q$ is rationally equivalent to zero on $S$. Now, if they are all torsion points, then the intersection 

$$\theta(C\times C)\cap J(C)_{tors}$$ 

is infinite, where 

$$\theta:  C\times C\to J(C)$$ 

the natural map given by 

$$(P,Q)\mapsto [P-Q]\;.$$

This is not possible by the Manin-Mumford theorem, as $J(C)$ is simple of dimension greater than 2 and does not contain abelian surfaces. Therefore, the kernel of $j_*$ contains infinitely many points of infinite order. All these points are obtained from $C\times C$. Consider the diagonal embedding of $C\to Sym^2 C$. Since we have 

$$j_*[P-P_0]=j_*[Q-P_0]$$

for $[P-Q]$  such that $j_*([P-Q])=0$ on the surface we have 
$[P-P_0], [Q-P_0]$ 
belonging to a certain fiber $F_w$ under the projection map of $R\to W$. There are a countably many such $[P_i-P_0], [Q_i-P_0]$ in the fiber of $F_w$. Suppose that the field of definitions of these zero cycles $L$ is a finite extension of $\QQ$, and hence the number of points on $C(L)$ is finite by Falting's Theorem \cite{F}. But $C(L)$ consists of all points of the form $P-P_0$ where $P_0$ is a $L$-rational point in $C$ and $P$ is in $R_i\cap C$ for a rational curve $R_i$ on $S$. Since there are infinitely many rational curves on $S$, there are infinitely many such points on $C(L)$ which contradicts Falting's Theorem, \cite{F}. So $L$ is not finite.

Hence $F_w$ generates a subgroup of $J(C)$  which is not defined over a finite extension $L$ of $\QQ$. In particular, it is not finitely generated. From the geometric side, since it is infinitely generated, it contains a curve $D$ in the Zariski closure of $F_{w}$. Now observe that $F_{w}$ is a countable union of Zariski closed subsets in $J(C)$. Each of these Zariski closed subsets are assumed to be proper subvarieties. Now, if all the countable number of points of the form $P-P_0$ in $F_{w}$ are in one of these Zariski closed subsets say $R$, then their clousre which is the curve $C$ is also in $R$.  Therefore, the curve $C$ generates the abelian variety $J(C)$.
Since $C$ is in $F_{w}$. On $\overline{F_w}$ we have :

$$Z_*j_*(z)=\deg(z)\Gamma_*(\omega)$$
and since $\deg(z)=0$, the above is zero.

Now suppose that all these countably many points $P_i-P_0$ are contained in a finitely many collection of these Zariski closed subsets $R_i$'s, then the Zariski closure $C$ is decomposed as 
$$C=\cup_i (C\cap R_i)$$
that forces $C\subset R_i$ for one $i$.

The hardest case is $C=\cup_i ( C\cap R_i)$ for countably many $R_i$. In that case, each $C\cap R_i$ consists of finitely many points of the form $P_{ij}-P_0$. But, on the other hand, the torsion points are dense in $J(C)$.  Consider the subgroup generated by the points of the form $P_{ij}-P_0$ in $C\cap R_i$. It is a subgroup of finite rank. We call it $H$. Then its Zariski closure is a finite union of translates of an abelian subvariety of $J(C)$. But this group is a  countable group (as the points of the form $P_{ij}-P_0$ are non-torsion), so it's Zarsiki closure contains a curve and it is positive dimensional. So by Mordell-Lang we have that it is a finite union of transltes of an abelian subvariety $B$ of $J(C)$. We know that $J(C)$ is simple and therefore $B=J(C)$. Therefore, $H$ is Zariski dense. Therefore, the subgroup generated by $F_{w}$ is Zariski dense as $H$ is contained in the algebraic  subgroup generated by $F_{w}$ and is equal to $J(C)$ by the previous argument. Therefore, the algebraic subgroup generated by $F_{w}$ is $J(C)$. In $F_{w}$ the homomorphism $Z_*j_*=0$, therefore on the algebraic subgroup generated by $F_{w}$ it is zero. Hence the claim follows.
\end{proof}

Now for any degree zero zero cycle $z_m$ on $S$, we write it as $z_m^+-z_m^-$, where $z_m^+=\sum_i P_1\cdots+P_k$ and $z_m^-=\sum_i P_{k+1}\cdots+P_{2k}$. So we have a tuple of points $(P_1,\cdots,P_{2k})$ in $S^{2k}$. We blow up $S$ at these points. Let $\tau:S'\to S$ be the blow-up, with exceptional divisors $E_i$ over the point $P_i$. Let $H$ belong to $\Pic(S)$, so that $L=\gamma^*(H)-n\sum E_i$ is ample on $S'$ and then consider a multiple of $L$ which is very ample. Then we argue as before saying that there exists a curve $C$ in the linear system of $mL$, such that the kernel of the map
$$J(C)\to \Alb(S)$$
is a simple abelian variety of sufficiently large dimension. Now $\tau(C)$ contains all the points $P_i$. Assume that the zero cycle $z_m$ is annihilated by $alb_S$, any lift of $z_m$ to $S'$ belongs to the kernel of $alb_{S'}$, since $\Alb(S')\to \Alb(S)$ is isomorphism. Let $z'$ be the lift of $z_m$ that is supported on $C$ so $z'$ belongs to $j_*(J(C))$. Now we apply the previous argument to $Z'=Z\circ \tau$, (noting that $Z'_*$ has a finite-dimensional image, as its image is contained in that of $Z_*$) between $S'$ and $S$, we get that
$Z_*(z_m)=Z'_*(z')=0$, so the homomorphism $Z_*$ factors through the albanese map.
\end{proof}

We consider the correspondence $\Gamma=\Delta_S-Graph(i)$. We prove that the image of $\Gamma_*$ is finite-dimensional.

\begin{theorem}
\label{theorem2}
Let $S$ be a K3 surface defined on $\bar \QQ$, which has an involution $i$ on it. Let $S/i$ be the projective plane $\PR^2$. Let the branch locus  of $i$ be a sextic curve $B$ in $\PR^2$. In addition, suppose that $S$ contains infinitely many rational curves. Consider the correspondence $\Delta_S-Graph(i)$ on $S\times S$. Then the image of the push-forward induced by this correspondence at the level of degree zero cycles modulo rational equivalence is finite dimensional.
\end{theorem}

\begin{proof}

We take a very ample line bundle $L$ on the surface $S/i=\PR^2$. Let the genus of a curve in the linear system of $L$ be $g$. We consider the exact sequence of sheaves
$$0\to\bcO(C)\to \bcO(\PR^2)\to \bcO(\PR^2)/\bcO(C)\to 0$$
Then tensoring with $\bcO(-C)$ we get
$$0\to \bcO(\PR^2)\to \bcO(-C)\to \bcO(-C)|_C\to 0\;.$$
This gives rise to the exact sequence of the global section of sheaves
$$0\to \CC\to H^0(\PR^2, L)\to H^0(C,L|_C)\to 0$$
this sequence is exact on the right because the irregularity of $S/i$ is zero. Now by the ampleness of $L$, we find that the degree of $L|_C$ is greater than zero, so we have by Riemann-Roch
$$\dim(H^0(C,L|_C))=g+n-1\;.$$

Here   $n$ is large by using the fact that the projective plane is  having anti-canonical divisor is ample  and we can choose $L$ to be of sufficiently large degree.

Hence, the dimension of $H^0(\PR^2,L)$ is equal to $g+n$.

Now for any $C$ in the linear system of $L$, we have its pre-image on $S$, say $\wt{C}$ is a two-fold branched cover of $C$. This double cover is ramified along the intersection of $C$ with the branch locus $B$ of the involution.

Let us consider a general point in $S^{g+n}$, say 

$$(s_1,\cdots,s_{g+n})$$

Let 
$$(\sigma'_1,\cdots,\sigma'_{g+n})$$ 
be the image of this point in $\PR^2=S/i$ under the quotient map. Then there exists a unique curve $C$ in $L$ such that its image under the quotient  contains all the points $\sigma'_i$, and its double cover contains $s_i$ for all $i$. Then the image of 
$$\sum_j s_j-\sum_j i(s_j)$$ 
depends on the element
$$alb_{\wt{C}}(\sum_j s_j-i(s_j))\in J(\wt{C})$$

So, the map from $S^{g+n}$ to $A_0(S)$ factors through the anti-invariant part  of the Jacobian fibrations $\bcJ^-$, where ${\bcC}$ is the universal family of curves on $\PR^2$ parametrized by the linear system $|L|$, and $\wt{\bcC}$ is its universal family of branched  double cover.

 So, the dimension of the anti-invariant part of the Jacobian  variety of the corresponding ramified double cover is
$$2g-1+B.C/2-g=g-1+B.C/2\;.$$

 Hence, the dimension of the anti-invariant part of the Jacobian fibration is equal to $2g+n-1+B.C/2$, while the dimension of $S^{g+n}$ is $2g+2n$. So, the fibers of the map
$$\phi: S^{g+n}\to \bcJ^- $$
are of dimension
$$2g+2n-(2g+n-1+B.C/2)=n+1-B.C/2$$
Now $B$ is a sextic curve in $\PR^2$ and is also of genus $10$. $n=C.(-K_{\PR^2})$. By the adjunction formula, $-K_{\PR^2}=3H$, where $H$ is a general linear hyperplane section of $\PR^2$. So $-K_{\PR^2}$ gives a linear system of cubics. Hence, a general curve $D$ in this linear system is an elliptic curve. Suppose that $C$ of degree $d$. Then we need to prove that, $n+1-B.C/2>0$. The above is
$$3d+1-3d=1$$

So, the general fibers of the above map are positive dimensional. Now consider the branch locus $B$ in $\PR^2$ and pull it back to $S$. The pullback is an ample  curve, denoted by the same letter $B$.  The curve $B$ is fixed point-wise by involution $i$. If we pullback this ample divisor $B$ to $S^{g+n}$ under the $j$-th projection $\pi_j$ and consider the divisor $\sum_j \pi_j^{-1}(B)$, this is an ample divisor on $S^{g+n}$. This ample divisor will intersect the general fiber of the map $\phi$. Denote this fiber by $F_s$ and there is a point $c$ such that 
$$(s_1,\cdots,c,\cdots,s_{g+n})\mapsto alb_{\wt{C}}(\sum_j s_j-i(s_j)+c-i(c))$$
for points $(s_1,\cdots,c,\cdots,s_{g+n})$ in $F_s$. Now, this cycle $c-i(c)=0$ as $B$ is fixed pointwise.

Hence, we have $\Gamma_*(S^{g+n-1})=\Gamma_*(S^{g+n})$, where $\Gamma=\Delta_S-Graph(i)$. This proves that the image of $\Gamma_*$ is finite dimensional by the following argument.

Now we prove by induction that $$\Gamma_*(S^{m_0})=\Gamma_*(S^m)$$ for all $m\geq g+n$.
So suppose that $$\Gamma_*(S^k)=\Gamma^*(S^{m_0})$$ for $k\geq g+n$, then we have to prove that $$\Gamma_*(S^{k+1})=\Gamma_*(S^{m_0})$$ So any element in $$\Gamma_*(S^{k+1})$$ can be written as  $$\Gamma_*(s_1+\cdots+s_{m_0})+\Gamma_*(s)$$ Now let $k-m_0=l$, then $m_0+1=k-l+1$. Since $k-l<k$, we have $k-l+1\leq k$, so $m_0+1\leq k$, so we have the cycle
$$\Gamma_*(s_1+\cdots+s_{m_0})+\Gamma_*(s)$$
supported on $S^k$, hence on $S^{m_0}$. Hence, we have that $$\Gamma_*(S^{m_0})=\Gamma_*(S^k)$$ for all $k$ greater or equal than $g+n$. Now any element $z$ in $A_0(S)$, can be written as a difference of two effective cycle $z^+,z^-$ of the same degree. Then we have
$$\Gamma_*(z)=\Gamma_*(z^+)-\Gamma_*(z_-)$$
and $\Gamma_(z_{\pm})$ belong to $\Gamma_*(S^{m_0})$. So let $\Gamma'$ be the correspondence on $S^{2m_0}\times S$ defined as
$$\sum_{l\leq m_0}(pr_i,pr_{S})^*\Gamma-\sum_{m_0\leq l\leq 2m_0}(pr_i,pr_{S})^* \Gamma$$
where $\pr_i$ is the $i$-th projection from $S^{m_0}$ to $S$, and $\pr_{S}$ is from $S^i\times S$ to the last copy of $S$. Then we have
$$\im(\Gamma_*)=\Gamma'_*(S^{2m_0})\;.$$

Therefore, the image of $(\Delta_{S}-Graph(i_{S}))_*$ is finite dimensional in the sense of Roitman.

\end{proof}
Since the irregularity of $S$ is zero, we have
\begin{theorem}
Involution $i_*$ acts as an identity in the group $A_0(S)$.
\end{theorem}

Since the quotient surface is the projective plane, the involution $i_*$ also acts as -1 on $A_0(S)$ and it follows that : 
\begin{theorem}
Let $S$ be a K3 surface defined over $\bar \QQ$, equipped with an involution $i$ such that the quotient map $S\to \PR^2\cong S/i$ is branched along a sextic curve and $S$ contains infinitely many rational curves. Then $A_0(S)=\{0\}$.
\end{theorem}

\begin{example}
According to \cite{BHT} any K3 surface defined over an algebraically closed field in characteristic 0 with the Picard group isomorphic to $\ZZ$ and generated by a degree two divisor contains infinitely many rational curves. In the paper by \cite{EJ}, the authors construct examples of K3 surfaces branched over $\PR^2$ along a sextic curve with Picard group isomorphic to $\ZZ$ generated by a degree two divisor. Those examples satisfy the Beilinson conjecture according to our theorem \ref{theorem2}.
\end{example}

\begin{remark}
In the two theorems, we have emphasized the use of infinitely many rational curves which is useful to apply \ref{theorem 1} and also note that \ref{theorem 1} is only applicable over $\bar \QQ$ and not over $\CC$ as it uses Falting's result intrinsically. So, one should not be tempted to generalize this result over $\CC$ and it is false over $\CC$.
\end{remark}

Data Availability : There is no associated data.

Conflict of interest: There is no conflict of interest.

\end{document}